\documentclass[12pt]{article}
\usepackage{amssymb}
\usepackage{amsthm}
\usepackage{amsmath}
\usepackage{graphicx}

\newtheorem{theorem}{Theorem}
\newtheorem{definition}[theorem]{Definition}
\newtheorem{lemma}[theorem]{Lemma}

\newtheorem{example}[theorem]{Example}
\newtheorem{corollary}[theorem]{Corollary}
\title{Residuation in modular lattices and posets}
\author{Ivan Chajda and Helmut L\"anger}
\date{}
\begin{document}
\footnotetext[1]{Support of the research of both authors by \"OAD, project CZ~04/2017, support of the research of the first author by IGA, project P\v rF~2018~012, and support of the research of the second author by the Austrian Science Fund (FWF), project I~1923-N25, is gratefully acknowledged.}
\maketitle
\begin{abstract}
We show that every complemented modular lattice can be converted into a left residuated lattice where the binary operations of multiplication and residuum are term operations. The concept of an operator left residuated poset was introduced by the authors recently. We show that every strongly modular poset with complementation as well as every strictly modular poset with complementation can be organized into an operator left residuated poset in such a way that the corresponding operators $M(x,y)$ and $R(x,y)$ can be expressed by means of the operators $L$ and $U$ in posets. We describe connections between the operator left residuation in these posets and the residuation in their lattice completion. We also present examples of strongly modular and strictly modular posets.
\end{abstract}
 


Residuated structures in general and residuated lattices in particular play an important role in algebraic semantics of non-classical logics, for example in the so-called fuzzy logic, see e.g.\ \cite B and \cite{GJKO}. In the literature there exist various definitions of the basic notions. We will use the following definition of a so-called integral left residuated l-groupoid (as defined in \cite{CLa}). For our reasons, we will use a shorter name.

\begin{definition}
A {\em left residuated lattice} is an algebra $\mathbf L=(L,\vee,\wedge,\odot,\rightarrow,0,1)$ of type $(2,2,2,2,0,0)$ satisfying the following conditions for all $x,y,z\in L$:
\begin{enumerate}
\item[{\rm(i)}] $(L,\vee,\wedge,0,1)$ is a bounded lattice,
\item[{\rm(ii)}] $x\odot1\approx1\odot x\approx x$,
\item[{\rm(iii)}] $x\odot y\leq z$ if and only if $x\leq y\rightarrow z$.
\end{enumerate}
Condition {\rm(iii)} is called {\em left adjointness}. If $\odot$ is commutative then $\mathbf L$ is called a {\em residuated lattice} and {\rm(iii)} is called {\em adjointness}. The left residuated lattice $\mathbf L$ is called {\em divisible} if it satisfies the identity
\[
(x\rightarrow y)\odot x\approx x\wedge y.
\]
\end{definition}

Let us note that in case $\odot$ is, moreover associative then it is called a {\em $t$-norm} (see e.g.\ \cite{GJKO}). In this case our notion of a residuated lattice coincides with that introduced in \cite B.

It is well-known that if $(B,\vee,\wedge,{}',0,1)$ is a Boolean algebra and we put
\begin{align*}
      x\odot y & :=x\wedge y, \\
x\rightarrow y & :=x'\vee y
\end{align*}
for all $x,y\in B$ then $(B,\vee,\wedge,\odot,\rightarrow,0,1)$ is a residuated lattice. Unfortunately, if $(L,\vee,\wedge,{}',0,1)$ is an orthomodular lattice then it cannot be converted into a residuated lattice. However, an orthomodular lattice can be converted into a left residuated one (as shown in \cite{CL17}). We will show that similar results can be obtained for complemented modular lattices.

Recall that a unary operation $'$ on a bounded lattice $(L,\vee,\wedge,0,1)$ is called
\begin{itemize}
\item {\em antitone} if $x\leq y$ implies $x'\geq y'$,
\item an {\em involution} if it satisfies the identity $x''\approx x$,
\item a {\em complementation} if it satisfies the identities $x\vee x'\approx1$ and $x\wedge x'\approx0$.
\end{itemize}
A {\em lattice} $(L,\vee,\wedge)$ id called {\em modular} if it satisfies the identity $(x\vee y)\wedge(x\vee z)\approx x\vee(y\wedge(x\vee z))$. A bounded lattice $(L,\vee,\wedge,{}',0,1)$ with a unary operation is called an
\begin{itemize}
\item {\em ortholattice} if $'$ is both a complementation and an antitone involution,
\item {\em orthomodular lattice} if it is an ortholattice satisfying the identity $x\vee((x\vee y)\wedge x')\approx x\vee y$.
\end{itemize}
It is well-known that every modular ortholattice is orthomodular. However, there exist modular lattices with a complementation which are not orthomodular as the following example shows:

\begin{example}
Let $\mathbf L=(L,\vee,\wedge,{}',0,1)$ denote the bounded modular lattice with an involution whose Hasse diagram is depicted in Fig.~1.

\vspace*{-2mm}

\[
\setlength{\unitlength}{7mm}
\begin{picture}(10,9)
\put(7,2){\circle*{.3}}
\put(3,4){\circle*{.3}}
\put(5,4){\circle*{.3}}
\put(7,4){\circle*{.3}}
\put(9,4){\circle*{.3}}
\put(1,6){\circle*{.3}}
\put(3,6){\circle*{.3}}
\put(5,6){\circle*{.3}}
\put(7,6){\circle*{.3}}
\put(3,8){\circle*{.3}}
\put(7,2){\line(-2,1)4}
\put(7,2){\line(-1,1)2}
\put(7,2){\line(0,1)4}
\put(7,2){\line(1,1)2}
\put(3,4){\line(-1,1)2}
\put(3,4){\line(1,1)2}
\put(3,8){\line(-1,-1)2}
\put(3,8){\line(0,-1)4}
\put(3,8){\line(1,-1)2}
\put(3,8){\line(2,-1)4}
\put(5,4){\line(-2,1)4}
\put(5,4){\line(1,1)2}
\put(5,6){\line(2,-1)4}
\put(7,4){\line(-2,1)4}
\put(7,6){\line(1,-1)2}
\put(6.85,1.3){$0$}
\put(2.85,3.3){$a$}
\put(5.5,3.85){$b$}
\put(7.35,3.85){$c$}
\put(9.35,3.85){$d$}
\put(.35,5.85){$c'$}
\put(2.35,5.85){$d'$}
\put(4.1,5.85){$b'$}
\put(7.45,5.85){$a'$}
\put(2.35,8.3){$1=0'$}
\put(4.35,.6){{\rm Fig.~1}}
\end{picture}
\]

\vspace*{-3mm}

It is evident that the involution $'$ is a complementation. However, $\mathbf L$ is not orthomodular since $'$ is not antitone: We have $b\leq c'$, but $c''=c\not\leq b'$.
\end{example}

We can state the following result showing how to organize such a lattice into left residuated one.

\begin{theorem}\label{th2}
Let $(L,\vee,\wedge,{}',0,1)$ be a complemented modular lattice and put
\begin{align*}
      x\odot y & :=(x\vee y')\wedge y, \\
x\rightarrow y & :=(x\wedge y)\vee x'
\end{align*}
for all $x,y\in L$. Then $(L,\vee,\wedge,\odot,\rightarrow,0,1)$ is a divisible left residuated lattice.
\end{theorem}

\begin{proof}
We have
\begin{align*}
                     0' & \approx0\vee0'\approx1, \\
                     1' & \approx1\wedge1'\approx0, \\
                x\odot1 & \approx(x\vee1')\wedge1\approx x, \\
               1\odot x & \approx(1\vee x')\wedge x\approx x, \\
(x\rightarrow y)\odot x & \approx(((x\wedge y)\vee x')\vee x')\wedge x\approx((x\wedge y)\vee x')\wedge x\approx(x\wedge y)\vee(x'\wedge x)\approx x\wedge y.
\end{align*}
Now let $a,b,c\in L$. If $a\odot b\leq c$ then
\begin{align*}
a & \leq a\vee b'=(b'\vee b)\wedge(a\vee b')=b'\vee(b\wedge(a\vee b'))=(b\wedge(a\vee b'))\vee b'= \\
  & =(b\wedge(a\odot b))\vee b'\leq(b\wedge c)\vee b'=b\rightarrow c.
\end{align*}
If, conversely, $a\leq b\rightarrow c$ then
\begin{align*}
a\odot b & =(a\vee b')\wedge b\leq((b\rightarrow c)\vee b')\wedge b=((b\wedge c)\vee b')\wedge b=(b\wedge c)\vee(b'\wedge b)=b\wedge c\leq \\
         & \leq c.
\end{align*}
\end{proof}

Now we turn our attention to complemented posets. We will investigate under which conditions a certain modification of a complemented modular poset can be organized into a so-called operator left residuated structure. At first, we recall all the necessary definitions and concepts.

Let $(P,\leq)$ be a poset. For arbitrary $A\subseteq P$ put
\begin{align*}
L(A) & :=\{x\in P\mid x\leq y\text{ for all }y\in A\}, \\
U(A) & :=\{x\in P\mid x\geq y\text{ for all }y\in A\}.
\end{align*}
We write $L(a,b)$, $L(a,A)$, $L(A,B)$, $LU(A)$ instead of $L(\{a,b\})$, $L(\{a\}\cup A)$, $L(A\cup B)$, $L(U(A))$, respectively. Analogously, we proceed in similar cases.

The following useful concept was introduced in \cite{CL18}:

\begin{definition}
An {\em operator left residuated poset} is an ordered seventuple $\mathbf P=(P,\leq,{}',M,R,0,$ $1)$ where $(P,\leq,{}',0,1)$ is a bounded poset with a unary operation and $M$ and $R$ are mappings from $P^2$ to $2^P$ satisfying the following conditions for all $x,y,z\in P$:
\begin{enumerate}
\item[{\rm(i)}] $M(x,1)\approx M(1,x)\approx L(x)$,
\item[{\rm(ii)}] $M(x,y)\subseteq L(z)$ if and only if $L(x)\subseteq R(y,z)$,
\item[{\rm(iii)}] $R(x,0)\approx L(x')$.
\end{enumerate}
Condition {\rm(ii)} is called {\em operator left adjointness}. If $M$ is commutative then {\rm(ii)} is called {\em operator adjointness} and $\mathbf P$ is called an {\em operator residuated poset}. The operator left residuated poset $\mathbf P$ is called {\em divisible} if it satisfies the LU-identity $M(R(x,y),x)\approx L(x,y)$.
\end{definition}

It is easy to prove that $x\leq y$ is equivalent to $R(x,y)=P$ (see e.g.\ \cite{CLa}).

A {\em poset} $(P,\leq)$ is called
\begin{itemize}
\item {\em modular} if $L(U(x,y),z)=LU(x,L(y,z))$ for all $x,y,z\in P$ with $x\leq z$,
\item{distributive} if it satisfies one of the following equivalent LU-identities:
\begin{align*}
 L(U(x,y),z) & \approx LU(L(x,z),L(y,z)), \\
LU(L(x,y),z) & \approx L(U(x,z),U(y,z)).
\end{align*}
\end{itemize}
A unary operation $'$ on a bounded poset $(P,\leq,0,1)$ is called a {\em complementation} if it satisfies the LU-identities $U(x,x')\approx\{1\}$ and $L(x,x')=\{0\}$. A bounded poset $(P,\leq,{}',0,1)$ with a unary operation is called
\begin{itemize}
\item an {\em orthoposet} if $'$ is a complementation and an antitone involution,
\item a {\em Boolean poset} if it is a distributive orthoposet.
\end{itemize}
It was shown in \cite{CLa} that every Boolean poset can be organized into an operator residuated one by means of $M(x,y):=L(x,y)$ and $R(x,y):=LU(x',y)$. Moreover, every pseudo-orthomodular poset can be converted into an operator left residuated poset. The aim of this paper is to find a suitable modification of the notion of modularity in posets such that these posets can be converted into operator left residuated posets. The first suitable candidate for such a modification is as follows:

\begin{definition}
A poset $(P,\leq)$ is called {\em strongly modular} if it satisfies the LU-identities
\begin{eqnarray}
L(U(x,y),U(x,z)) & \approx & LU(x,L(y,U(x,z))),\label{equ3} \\
L(U(L(x,z),y),z) & \approx & LU(L(x,z),L(y,z)).\label{equ4}
\end{eqnarray}
\end{definition}

It is easy to see that strong modularity implies modularity.

Of course, every modular lattice as well as every Boolean poset is a strongly modular poset. An example of a strongly modular poset which is neither a lattice nor a Boolean poset is the direct product $\mathbf P=\mathbf L\times\mathbf B$ where $\mathbf L$ is the lattice depicted in Fig.~1 (as a poset) and $\mathbf B$ is the poset visualized in Fig.~2.

\vspace*{3mm}

\[
\setlength{\unitlength}{7mm}
\begin{picture}(6,8)
\put(3,0){\circle*{.3}}
\put(0,2){\circle*{.3}}
\put(2,2){\circle*{.3}}
\put(4,2){\circle*{.3}}
\put(6,2){\circle*{.3}}
\put(0,4){\circle*{.3}}
\put(6,4){\circle*{.3}}
\put(0,6){\circle*{.3}}
\put(2,6){\circle*{.3}}
\put(4,6){\circle*{.3}}
\put(6,6){\circle*{.3}}
\put(3,8){\circle*{.3}}
\put(3,0){\line(-3,2)3}
\put(3,0){\line(-1,2)1}
\put(3,0){\line(1,2)1}
\put(3,0){\line(3,2)3}
\put(3,8){\line(-3,-2)3}
\put(3,8){\line(-1,-2)1}
\put(3,8){\line(1,-2)1}
\put(3,8){\line(3,-2)3}
\put(0,2){\line(0,1)4}
\put(6,2){\line(0,1)4}
\put(0,4){\line(1,1)2}
\put(0,2){\line(1,1)4}
\put(2,2){\line(1,1)4}
\put(4,2){\line(1,1)2}
\put(2,2){\line(-1,1)2}
\put(4,2){\line(-1,1)4}
\put(6,2){\line(-1,1)4}
\put(6,4){\line(-1,1)2}
\put(2.85,-.75){$0$}
\put(-.6,1.9){$a$}
\put(1.2,1.9){$b$}
\put(4.45,1.9){$c$}
\put(6.4,1.9){$d$}
\put(-.6,3.9){$e$}
\put(6.4,3.9){$e'$}
\put(-.7,5.9){$d'$}
\put(1.2,5.9){$c'$}
\put(4.45,5.9){$b'$}
\put(6.4,5.9){$a'$}
\put(2.85,8.4){$1$}
\put(2.3,-2){{\rm Fig.\ 2}}
\end{picture}
\]

\vspace*{13mm}

Moreover this poset is not orthomodular.

Now we can prove the following theorem.

\begin{theorem}\label{th1}
Let $(P,\leq,{}',0,1)$ be a bounded strongly modular poset with a complementation and put
\begin{align*}
M(x,y) & :=L(U(x,y'),y), \\
R(x,y) & :=LU(L(x,y),x')
\end{align*}
for all $x,y\in P$. Then $(P,\leq,{}',M,R,0,1)$ is a divisible operator left residuated poset.
\end{theorem}

\begin{proof}
We have
\begin{align*}
      U(0') & \approx U(0,0')=\{1\},\text{ i.e.\ }0'\approx1, \\
      L(1') & \approx L(1,1')=\{0\},\text{ i.e.\ }1'\approx0, \\
     M(x,1) & \approx L(U(x,1'),1)\approx L(x), \\
     M(1,x) & \approx L(U(1,x'),x)\approx L(x), \\
     R(x,0) & \approx LU(L(x,0),x')\approx L(x'), \\
M(R(x,y),x) & \approx L(U(LU(L(x,y),x'),x'),x)\approx L(ULU(L(x,y),x')\cap U(x'),x)\approx \\
            & \approx L(U(L(x,y),x')\cap U(x'),x)\approx L(U(L(x,y),x'),x)\approx \\
            & \approx L(U(L(y,x),x'),x)\approx LU(L(y,x),L(x',x))\approx LUL(y,x)\approx L(y,x)\approx \\
            & \approx L(x,y).
\end{align*}
Let $a,b,c\in P$. If $M(a,b)\subseteq L(c)$ then, according to (\ref{equ3}),
\begin{align*}
L(a) & =LU(a)\subseteq LU(b',a)=L(U(b',b),U(b',a))=LU(b',L(b,U(b',a)))= \\
     & =LU(L(b)\cap L(U(a,b'),b),b')=LU(L(b)\cap M(a,b),b')\subseteq LU(L(b)\cap L(c),b')= \\
		 & =LU(L(b,c),b')=R(b,c).
\end{align*}
If, conversely, $L(a)\subseteq R(b,c)$ then, according to (\ref{equ4}),
\begin{align*}
M(a,b) & =L(U(a,b'),b)=L(U(L(a),b'),b)\subseteq L(U(R(b,c),b'),b)= \\
       & =L(U(LU(L(b,c),b'),b'),b)=L(ULU(L(b,c),b')\cap U(b'),b)= \\
			 & =L(U(L(b,c),b')\cap U(b'),b)=L(U(L(b,c),b'),b)=L(U(L(c,b),b'),b)= \\
			 & =LU(L(c,b),L(b',b))=LUL(c,b)=L(c,b)\subseteq L(c).
\end{align*}
\end{proof}

Another appropriate modification of the notion of a modular poset is the following:

Let $\mathbf P=(P,\leq)$ be a poset, $a\in P$ and $A,B\subseteq P$. We denote by $A\leq B$ the fact that $x\leq y$ for all $x\in A$ and $y\in B$. Instead of $\{a\}\leq A$ or $A\leq\{a\}$ we simply write $a\leq A$ or $A\leq a$, respectively.

\begin{definition}
A poset $(P,\leq)$ is called {\em strictly modular} if for all $x,y,z\in P$ and $X,Z\subseteq P$ we have
\begin{eqnarray}
& & x\leq Z\text{ implies }L(U(x,y),Z)=LU(x,L(y,Z)),\label{equ1} \\
& & L(X)\leq z\text{ implies }L(U(L(X),y),z)=LU(L(X),L(y,z)).\label{equ2}
\end{eqnarray}
\end{definition}

It is easy to see that strict modularity implies modularity.

\begin{lemma}
Let $\mathbf L=(L,\vee,\wedge)$ be a complete modular lattice. Then $\mathbf L$ is strictly modular.
\end{lemma}

\begin{proof}
Let $a,b,c\in L$ and $A,C\subseteq L$. If $a\leq C$ then
\begin{align*}
L(U(a,b),C) & =L(U(a\vee b),C)=L(a\vee b,\bigwedge C)=L((a\vee b)\wedge\bigwedge C)= \\
& =L(a\vee(b\wedge\bigwedge C))=LU(a\vee(b\wedge\bigwedge C))=LU(a,b\wedge\bigwedge C)= \\
& =LU(a,L(b\wedge\bigwedge C))=LU(a,L(b,\bigwedge C))=LU(a,L(b,C)).
\end{align*}
If $L(A)\leq c$ then
\begin{align*}
L(U(L(A),b),c) & =L(U(\bigwedge A,b),c)=L(U(\bigwedge A\vee b),c)=L(\bigwedge A\vee b,c)= \\
& =L((\bigwedge A\vee b)\wedge c)=L(\bigwedge A\vee(b\wedge c))=LU(\bigwedge A\vee(b\wedge c))= \\
& =LU(\bigwedge A,b\wedge c)=LU(L(A),L(b,c)).
\end{align*}
\end{proof}

Hence every finite modular lattice is a strictly modular poset. The poset $\mathbf P_6$ visualized in Fig.~3

\vspace*{-7mm}

\[
\setlength{\unitlength}{7mm}
\begin{picture}(6,9)
\put(3,2){\circle*{.3}}
\put(1,4){\circle*{.3}}
\put(5,4){\circle*{.3}}
\put(1,6){\circle*{.3}}
\put(5,6){\circle*{.3}}
\put(3,8){\circle*{.3}}
\put(3,2){\line(-1,1)2}
\put(3,2){\line(1,1)2}
\put(1,6){\line(0,-1)2}
\put(1,6){\line(2,-1)4}
\put(1,6){\line(1,1)2}
\put(5,6){\line(0,-1)2}
\put(5,6){\line(-2,-1)4}
\put(5,6){\line(-1,1)2}
\put(2.875,1.25){$0$}
\put(.35,3.85){$a$}
\put(5.4,3.85){$b$}
\put(.35,5.85){$c$}
\put(5.4,5.85){$d$}
\put(2.85,8.4){$1$}
\put(2.2,.3){{\rm Fig.~3}}
\end{picture}
\]

\vspace*{-3mm}

is also strictly modular and even distributive. Hence, for a finite modular non-distributive lattice $\mathbf L$, the direct product $\mathbf L\times\mathbf P_6$ is a strictly modular bounded poset which is neither a lattice nor distributive.

\begin{theorem}\label{th4}
Let $\mathbf P=(P,\leq,{}',0,1)$ be a bounded strictly modular poset with complementation and put
\begin{align*}
M(x,y) & :=L(U(x,y'),y), \\
R(x,y) & :=LU(L(x,y),x')
\end{align*}
for all $x,y\in P$. Then $(P,\leq,{}',M,R,0,1)$ is a divisible operator left residuated poset.
\end{theorem}

\begin{proof}
We have $M(x,1)\approx M(1,x)\approx L(x)$ and $R(x,0)\approx L(x')$ as in the proof of Theorem~\ref{th1}. Moreover,
\begin{align*}
M(R(x,y),x) & \approx L(U(LU(L(x,y),x'),x'),x)\approx L(ULU(L(x,y),x')\cap U(x'),x)\approx \\
            & \approx L(U(L(x,y),x')\cap U(x'),x)\approx L(U(L(x,y),x'),x)\approx \\
						& \approx LU(L(x,y),L(x',x))\approx LUL(x,y)\approx L(x,y).
\end{align*}
Let $a,b,c\in P$. If $M(a,b)\subseteq L(c)$ then, according to (\ref{equ1}),
\begin{align*}
L(a) & =LU(a)\subseteq LU(a,b')=L(U(b',b),U(a,b'))=LU(b',L(b,U(a,b')))= \\
     & =LU(L(b)\cap L(U(a,b'),b),b')=LU(L(b)\cap M(a,b),b')\subseteq LU(L(b)\cap L(c),b')= \\
		 & =LU(L(b,c),b')=R(b,c).
\end{align*}
Conversely, if $L(a)\subseteq R(b,c)$ then, according to (\ref{equ2}),
\begin{align*}
M(a,b) & =L(U(a,b'),b)=L(U(L(a),b'),b)\subseteq L(U(R(b,c),b'),b)= \\
       & =L(U(LU(L(b,c),b'),b'),b)=L(ULU(L(b,c),b')\cap U(b'),b)= \\
       & =L(U(L(b,c),b')\cap U(b'),b)=L(U(L(b,c),b'),b)=LU(L(b,c),L(b',b))= \\
			 & =LUL(b,c)=L(b,c)\subseteq L(c).
\end{align*}
\end{proof}

Both strong modularity and strict modularity are generalizations of the concept of modularity. The question arises if the Dedekind-MacNeille completion of such posets is a modular lattice. In the remaining part of the paper we present some results connected with this question.

Let $\mathbf P=(P,\leq)$ be a poset. Put $D(\mathbf P):=\{A\subseteq P\mid LU(A)=A\}$. Then $\mathbf D(\mathbf P):=(D(\mathbf P),\subseteq)$ is a complete lattice which can be considered as an extension of $\mathbf P$ if one identifies $x$ with $L(x)$ for every $x\in P$. The lattice $\mathbf D(\mathbf P)$ is called the {\em Dedekind-MacNeille completion} of $\mathbf P$. We have
\begin{align*}
  A\vee B & =LU(A,B), \\
A\wedge B & =A\cap B
\end{align*}
for all $A,B\in D(\mathbf P)$. Let $\mathbf D_0(\mathbf P)=(D_0(\mathbf P),\subseteq)$ denote the sublattice of $\mathbf D(\mathbf P)$ generated by $P$ (which is identified with $\{L(x)\mid x\in P\}$).

The following theorem uses a construction which was also used in \cite{H96}.

\begin{theorem}\label{th3}
Let $\mathbf P=(P,\leq,{}',0,1)$ be an orthoposet. Then $'$ can be extended to an orthocomplementation $^*$ on $\mathbf D_0(\mathbf P)$ and further to an orthocomplementation $^*$ on $\mathbf D(\mathbf P)$.
\end{theorem}

\begin{proof}
For every $A\in D(\mathbf P)$ we define
\begin{align*}
 A' & :=\{x'\mid x\in A\}, \\
A^* & :=L(A').
\end{align*}
Then for arbitrary $a\in P$ and $A,B\in D(\mathbf P)$ we have
\begin{itemize}
\item $L(A')=(U(A))'$, $U(A')=(L(A))'$,
\item $A\subseteq B$ implies $A^*=L(A')\supseteq L(B')=B^*$,
\item $A^{**}=L((L(A'))')=L((U(A))'')=LU(A)=A$,
\item $A\vee A^*=LU(A,L(A'))=L(U(A)\cap UL(A'))=L(U(A)\cap(LU(A))')=L(U(A)\cap A')=L(1)=P$,
\item $A\cap A^*=A\cap L(A')=\{0\}$,
\item $(L(a))^*=L((L(a))')=LU(a')=L(a')$.
\end{itemize}
This shows that $^*$ is an extension of $'$ to an orthocomplementation on $\mathbf D(\mathbf P)$. Since $^*$ is an antitone involution, the De Morgan laws hold and since $^*$ is an extension of $'$ we have that $P$ is closed under $^*$. Because of the De Morgan laws, $D_0(\mathbf P)$ is closed under $^*$, too, and therefore $^*$, restricted to $D_0(\mathbf P)$, is an orthocomplementation on $\mathbf D_0(\mathbf P)$.
\end{proof}

In the following we show which influence the modularity of the lattices $\mathbf D_0(\mathbf P)$ or $\mathbf D(\mathbf P)$ for some poset $\mathbf P$ has on the structure of $\mathbf P$.

\begin{lemma}\label{lem1}
Let $\mathbf P$ be a poset and assume $\mathbf D_0(\mathbf P)$ to be modular. Then $\mathbf P$ is strongly modular.
\end{lemma}

\begin{proof}
We have
\begin{align*}
L(U(x,y),U(x,z)) & \approx LU(x,y)\cap LU(x,z)\approx(L(x)\vee L(y))\cap(L(x)\vee L(z))\approx \\
                 & \approx L(x)\vee(L(y)\cap(L(x)\vee L(z)))\approx L(x)\vee(L(y)\cap LU(x,z))\approx \\
                 & \approx L(x)\vee L(y,U(x,z))\approx LU(x,L(y,U(x,z)))
\end{align*}
and
\begin{align*}
L(U(L(x,z),y),z) & \approx LU(L(x,z),y)\cap L(z)\approx(L(x,z)\vee L(y))\cap L(z)\approx \\
                 & \approx((L(x)\cap L(z))\vee L(y))\cap L(z)\approx(L(x)\cap L(z))\vee(L(y)\cap L(z))\approx \\
                 & \approx L(x,z)\vee L(y,z)\approx LU(L(x,z),L(y,z)).
\end{align*}
\end{proof}

In \cite{H95} a similar result was obtained for distributivity. Now we can prove an analogous result for $\mathbf D(\mathbf P)$ instead of $\mathbf D_0(\mathbf P)$.

\begin{lemma}\label{lem2}
Let $\mathbf P$ be a poset and assume $\mathbf D(\mathbf P)$ to be modular. Then $\mathbf P$ is strictly modular.
\end{lemma}

\begin{proof}
Let $a,b,c\in P$ and $A,C\subseteq P$. If $a\leq C$ then
\begin{align*}
L(U(a,b),C) & =LU(a,b)\cap L(C)=(L(a)\vee L(b))\cap L(C)=L(a)\vee(L(b)\cap L(C))= \\
            & =L(a)\vee L(b,c)=LU(a,L(b,C))
\end{align*}
and if $L(A)\leq c$ then
\begin{align*}
L(U(L(A),b),c) & =LU(L(A),b)\cap L(c)=(L(A)\vee L(b))\cap L(c)= \\
& =L(A)\vee(L(b)\cap L(c))=L(A)\vee L(b,c)=LU(L(A),L(b,c)).
\end{align*}
\end{proof}

It is well-known that the Dedekind-MacNeille completion of an orthomodular poset need not be an orthomodular lattice. However, we can prove the following

\begin{corollary}
Let $\mathbf P=(P,\leq,{}',0,1)$ be an orthoposet and assume $\mathbf D_0(\mathbf P)$ to be modular. Then
\begin{enumerate}
\item[{\rm(i)}] $\mathbf P$ is strongly modular,
\item[{\rm(ii)}] if one defines
\begin{align*}
M(x,y) & :=L(U(x,y'),y), \\
R(x,y) & :=LU(L(x,y),x')
\end{align*}
for all $x,y\in P$ then $(P,\leq,{}',M,R,0,1)$ is a divisible operator left residuated poset,
\item[{\rm(iii)}] there exists an extension $^*$ of $'$ to $D_0(\mathbf P)$ such that $(D_0(\mathbf P),\vee,\cap,{}^*,\{0\},P)$ is a modular ortholattice and hence an orthomodular lattice, and if we define
\begin{align*}
      A\odot B & :=(A\vee B^*)\cap B, \\
A\rightarrow B & :=(A\cap B)\vee A^*
\end{align*}
for all $A,B\in D_0(\mathbf P)$ then $(D_0(\mathbf P),\vee,\cap,\odot,\rightarrow,0,1)$ is a divisible left residuated lattice.
\end{enumerate}
\end{corollary}

\begin{proof}
This follows from Theorems~\ref{th2}, \ref{th1} and \ref{th3} and Lemma~\ref{lem1}.
\end{proof}

\begin{corollary}
Let $\mathbf P=(P,\leq,{}',0,1)$ be an orthoposet and assume $\mathbf D(\mathbf P)$ to be modular. Then
\begin{enumerate}
\item[{\rm(i)}] $\mathbf P$ is strictly modular,
\item[{\rm(ii)}] if one defines
\begin{align*}
M(x,y) & :=L(U(x,y'),y), \\
R(x,y) & :=LU(L(x,y),x')
\end{align*}
for all $x,y\in P$ then $(P,\leq,{}',M,R,0,1)$ is a divisible operator left residuated poset,
\item[{\rm(iii)}] there exists an extension $^*$ of $'$ to $D(\mathbf P)$ such that $(D(\mathbf P),\vee,\cap,{}^*,\{0\},P)$ is a modular ortholattice and hence an orthomodular lattice, and if we define
\begin{align*}
      A\odot B & :=(A\vee B^*)\cap B, \\
A\rightarrow B & :=(A\cap B)\vee A^*
\end{align*}
for all $A,B\in D(\mathbf P)$ then $(D(\mathbf P),\vee,\cap,\odot,\rightarrow,0,1)$ is a divisible left residuated lattice.
\end{enumerate}
\end{corollary}

\begin{proof}
This follows from Theorems~\ref{th2}, \ref{th4} and \ref{th3} and Lemma~\ref{lem2}.
\end{proof}

Authors' addresses:

Ivan Chajda \\
Palack\'y University Olomouc \\
Faculty of Science \\
Department of Algebra and Geometry \\
17.\ listopadu 12 \\
771 46 Olomouc \\
Czech Republic \\
ivan.chajda@upol.cz

Helmut L\"anger \\
TU Wien \\
Faculty of Mathematics and Geoinformation \\
Institute of Discrete Mathematics and Geometry \\
Wiedner Hauptstra\ss e 8-10 \\
1040 Vienna \\
Austria, and \\
Palack\'y University Olomouc \\
Faculty of Science \\
Department of Algebra and Geometry \\
17.\ listopadu 12 \\
771 46 Olomouc \\
Czech Republic \\
helmut.laenger@tuwien.ac.at
\end{document}